\documentclass[reqno]{amsart}

\usepackage{amssymb,amsmath,amsthm,xcolor,pdfsync,enumerate,hyperref,cleveref,fixltx2e}

\allowdisplaybreaks

\newtheorem{thrm}{Theorem}
\newtheorem{cor}{Corollary}
\newtheorem{lem}{Lemma}
\newtheorem{prop}{Proposition}

\theoremstyle{definition}

\newtheorem{rem}{Remark}

\crefrangeformat{equation}{#3(#1)#4--#5(#2)#6}

\crefname{thrm}{Theorem}{Theorems}
\crefname{lem}{Lemma}{Lemmas}
\crefname{cor}{Corollary}{Corollaries}
\crefname{prop}{Proposition}{Propositions}
\crefname{defn}{Definition}{Definitions}
\crefname{exm}{Example}{Examples}
\crefname{rem}{Remark}{Remarks}
\crefname{section}{Section}{Sections}
\crefname{equation}{}{}

\renewcommand{\iff}{\Leftrightarrow}

\newcommand{\impl}%
{%
 \Rightarrow
}

\newcommand{\id}{\mathrm{id}}

\newcommand{\dom}[1]%
{%
 \operatorname{\mathrm{dom}}{#1}%
}

\newcommand{\ran}[1]%
{%
 \operatorname{\mathrm{ran}}{#1}%
}

\newcommand{\cC}%
{%
 \mathcal C%
}

\newcommand{\ob}[1]%
{%
 \operatorname{\mathrm{Ob}}{#1}%
}

\newcommand{\mor}[2]%
{%
 \operatorname{\mathrm{Mor}}(#1,#2)%
}

\newcommand{\ch}[1]%
{%
 \operatorname{\mathrm{char}}{#1}%
}

\begin{document}

\title{Jordan derivations of finitary incidence rings}
\author{Mykola Khrypchenko}
\address{Universidade Federal de Santa Catarina, Florian\'opolis, SC, Brazil, CEP: 88040--900}
\email{nskhripchenko@gmail.com}

\begin{abstract}
 Let $P$ be a preordered set, $R$ a ring and $FI(P,R)$ the finitary incidence ring of $P$ over $R$~\cite{Khripchenko10-quasi}. We find a criterion for all Jordan derivations of $FI(P,R)$ to be derivations and generalize Theorem 3.3 from~\cite{XiaoJDerIAlg}. In particular, we prove that each Jordan derivation of the ring $RFM_I(R)$ of row-finite $I\times I$-matrices over $R$ is a derivation, if $|I|>1$.
\end{abstract}

\maketitle

\section*{Introduction}
Given a ring $R$ and $n\ge 2$, it follows from the classical results by Jacobson and Rickart (\cite[Theorems 8 and 22]{Jacobson-Rickart50}) that every Jordan derivation of the ring $M_n(R)$ of all $n\times n$-matrices is a derivation. Benkovi\v c proved that the same is true for the algebra $T_n(R)$ of upper triangular $n\times n$-matrices over a commutative unital ring $R$ of characteristic different from $2$ (\cite[Corollary 1.2]{Benkovic05}). In~\cite{Ghosseiri07} the author described Jordan derivations of a subring of $M_n(R)$ which contains $T_n(R)$, under the same assumption on $R$ (see~\cite[Theorem~1]{Ghosseiri07}). In particular, the Benkovi\v c's result follows from this description (\cite[Corollary~2]{Ghosseiri07}).

The algebra $T_n(R)$, $n\ge 2$, is a particular case of the triangular algebra $Tri(A,M,B)$, where $A,B$ are unital $R$-algebras and $M$ is an $(A,B)$-bimodule. Jordan derivations of such algebras over a ring of characteristics different from $2$ were studied in~\cite{Zhang-Yu06}. Theorem~2.1 from~\cite{Zhang-Yu06} says that each Jordan derivation of $Tri(A,M,B)$ is a derivation, if $M$ is faithful as a left $A$-module and as a right $B$-module.

Benkovi\v c and \v Sirovnik generalized~\cite[Theorem~1]{Ghosseiri07} to unital algebras admitting nontrivial idempotents with certain conditions (\cite[Theorem~4.1]{Benkovic-Sirovnik12}). This also permitted to obtain~\cite[Theorem~2.1]{Zhang-Yu06} as a consequence (see~\cite[Corollary~4.3]{Benkovic-Sirovnik12}).

Another example of a triangular algebra is the incidence algebra $I(P,R)$ of a finite poset $P$, $|P|\ge 2$, over a commutative ring $R$. Indeed, fixing an enumeration $P=\{x_1,\dots,x_n\}$ with $x_i\le x_j\impl i\le j$ and choosing $1\le m\le n$, one can show that $I(P,R)$ is isomorphic to $Tri(I(Q,R),M,I(P\setminus Q,R))$, where $Q=\{x_1,\dots,x_m\}$ and $M$ consists of the functions $f\in I(P,R)$, such that $f(x,y)=0$, whenever $x,y\in Q$ or $x,y\in P\setminus Q$. However,~\cite[Theorem~2.1]{Zhang-Yu06} cannot be directly applied to $I(P,R)$, because $M$ is not always faithful as a left $I(Q,R)$-module or as a right $I(P\setminus Q,R)$-module. For instance, when $P$ is an anti-chain, the bimodule $M$ is trivial for any choice of $Q$. Nevertheless, using direct computations, Xiao proved in~\cite{XiaoJDerIAlg} that every Jordan derivation of $I(P,R)$ is a derivation, provided that $P$ is a locally finite preordered set and $\ch R\ne 2$ (\cite[Theorem~3.3]{XiaoJDerIAlg}).

Finitary incidence algebras appeared in~\cite{Khripchenko-Novikov09} as a generalization of incidence algebras to the non-locally finite case. Initially they were defined, when $P$ is a poset and $R$ is a field. Then the definition was extended in~\cite{Khripchenko10-quasi} to the so-called partially ordered categories~\cite{Khripchenko10-quasi} and, in particular, to preordered sets and rings, giving thus the finitary incidence ring $FI(P,R)$ of $P$ over $R$.

The aim of this article is to generalize~\cite[Theorem~3.3]{XiaoJDerIAlg} to arbitrary $P$ and $R$, considering $FI(P,R)$ instead of $I(P,R)$. To avoid the restriction on the characteristics of $R$, we use the Jacobson-Rickart's definition of a Jordan derivation~\cite[(26)--(27)]{Jacobson-Rickart50}.

In Section~\ref{sec-tech} we prove some results on Jordan derivations which take place in an arbitrary ring (see \cref{lem-ed(r)f,lem-ed(e)f+ed(f)f,lem-d_e,lem-ed(rs)f}).

In Section~\ref{sec-tech-fin} we restrict ourselves to rings $R$ admitting a set of pairwise orthogonal idempotents which impose certain conditions on the multiplication in $R$. One should note that this class of rings contains the incidence ring $FI(P,R)$ as well as the ring $RFM_I(R)$ of row-finite matrices over $R$ whose entries are indexed by the pairs of elements of $I$. The main result of Section~\ref{sec-tech-fin} is Lemma~\ref{lem-d'-der} which permits us to construct a derivation $d'$ of $R$ from a Jordan derivation $d$ of $R$ under some assumptions on $d$.

In Section~\ref{sec-FI(C)} we apply the lemmas from \cref{sec-tech,sec-tech-fin,} to the finitary incidence ring $FI(\cC)$ of a partially ordered category $\cC$. Proposition~\ref{prop-d-der-iff-d_x-der} says that a Jordan derivation $d$ of $FI(\cC)$ is a derivation if and only if the ``restriction'' $d_x$ of $d$ to $\mor xx$ is a derivation for each $x\in\ob\cC$. Using this proposition and~\cite[Theorem~2.1]{Zhang-Yu06} under certain restrictions on $\cC$, we reduce the question whether all Jordan derivations of $FI(\cC)$ are derivations to the same question in the rings $\mor xx$, where $x$ runs through the set of isolated objects of $\cC$ (see Theorem~\ref{thrm-crit-jder-FI(C)-is-der}).

In Section~\ref{sec-RFM} we study Jordan derivations of the ring $RFM_I(R)$. We prove, with the help of Lemma~\ref{lem-d'-der} and the result by Jacobson and Rickart on Jordan derivations of $M_n(R)$, that each Jordan derivation of $RFM_I(R)$ is a derivation, when $|I|>1$ (see Theorem~\ref{thrm-jder-of-RFM_I-for-|I|>1}). Thus, we extend the Jacobson-Rickart's result to a ring of infinite matrices.

Section~\ref{sec-FI(P,R)} is an easy application of \cref{thrm-crit-jder-FI(C)-is-der,thrm-jder-of-RFM_I-for-|I|>1}. We show that each Jordan derivation of $FI(P,R)$ is a derivation, when $P$ has no isolated elements. Otherwise the question whether each Jordan derivation of $FI(P,R)$ is a derivation reduces to the same question for the ring $R$ (see Theorem~\ref{thrm-crit-jder-FI(P,R)-is-der}). As a remark we obtain a generalization of~\cite[Theorem~3.3]{XiaoJDerIAlg} (see~\cref{rem-R-linear}).

\section{Preliminaries}\label{sec-prelim}
Let $R$ be a ring and $d$ an additive map $d:R\to R$. 	Following~\cite{Jacobson-Rickart50}, we shall call $d$ a {\it Jordan derivation} of $R$, if it satisfies
\begin{align}
 d(r^2)&=d(r)r+rd(r),\label{eq-d(r^2)}\\
 d(rsr)&=d(r)sr+rd(s)r+rsd(r)\label{eq-d(rsr)}
\end{align}
for arbitrary $r,s\in R$. Obviously, a usual derivation of $R$, that is an additive map $d:R\to R$ with
\begin{align*}
 d(rs)=d(r)s+rd(s)
\end{align*}
for all $r,s\in R$, is a Jordan derivation of $R$.

Let $d$ be a Jordan derivation of $R$. Applying~\eqref{eq-d(r^2)} to the sum $r+s$, we immediately get
\begin{align}\label{eq-d(rs+sr)}
 d(rs+sr)=d(r)s+rd(s)+d(s)r+sd(r)
\end{align}
for all $r,s\in R$. Equality~\eqref{eq-d(rs+sr)} is often used as a definition of a Jordan derivation of $R$ (see, for example,~\cite{Herstein57}). It is in fact equivalent to~\eqref{eq-d(r^2)}, when $\ch R\ne 2$. Indeed, to get~\eqref{eq-d(r^2)}, one simply substitutes $s=r$ into~\eqref{eq-d(rs+sr)} and ``divides'' both sides by $2$. Condition~\eqref{eq-d(rsr)} also follows from~\eqref{eq-d(rs+sr)} in this case: to prove it, one sets $s=rs'+s'r$ in~\eqref{eq-d(rs+sr)}, where $s'$ is an arbitrary element of $R$ (see~\cite[Lemma 3.1]{Herstein57}).

We shall also use the following formula by Herstein (\cite[Lemma 3.2]{Herstein57}):
\begin{align}\label{eq-d(rst+tsr)}
 d(rst+tsr)=d(r)st+rd(s)t+rsd(t)+d(t)sr+td(s)r+tsd(r),
\end{align}
which is a consequence of~\eqref{eq-d(rsr)} with $r$ replaced by $r+t$.

Let $\cC$ be a preadditive small category. We shall call $\cC$ a {\it pocategory}, when $\ob\cC$ admits a partial order $\le$. Denote by $I(\cC)$ the set of the formal sums
\begin{align}\label{eq-formal-sum}
 \alpha=\sum_{x\le y}\alpha_{xy}[x,y],
\end{align}
where $x,y\in\ob\cC$, $[x,y]=\{z\in\ob\cC\mid x\le z\le y\}$ and $\alpha_{xy}\in\mor xy$. Clearly, $I(\cC)$ is an abelian group under the addition, which naturally comes from the addition of morphisms in $\cC$. Without loss of generality we shall also consider the series $\alpha$ of the form~\eqref{eq-formal-sum}, whose indices run through a subset $X$ of the segments of $\ob\cC$, meaning that $\alpha_{xy}$ is the zero $0_{xy}$ of $\mor xy$ for $[x,y]\not\in X$.

The sum~\eqref{eq-formal-sum} is called a {\it finitary series}, whenever for any pair of $x,y\in\ob\cC$ with $x<y$ there exists only a finite number of $u,v\in\ob\cC$, such that $x\le u<v\le y$ and $\alpha_{uv}\ne 0_{uv}$. The set of finitary series will be denoted by $FI(\cC)$. Note that $FI(\cC)$ is an additive subgroup of $I(\cC)$. Moreover, $FI(\cC)$ is closed under the convolution of the series:
\begin{align}\label{eq-convolution}
 \alpha\beta=\sum_{x\le y}\left(\sum_{x\le z\le y}\alpha_{xz}\beta_{zy}\right)[x,y]
\end{align}
for $\alpha,\beta\in FI(\cC)$. Thus, $FI(\cC)$ is a ring, called the {\it finitary incidence ring of $\cC$} \footnote{In~\cite{Khripchenko10-quasi} by a pocategory one means a preadditive small category $\cC$ with a partial order $\le$ on $\ob\cC$ satisfying $x\le y\iff\mor xy\ne 0_{xy}$. It turns out that the last property is superfluous for the definition of $FI(\cC)$.}


Given $x\in\ob\cC$, denote by $e_x$ the element $\id_x[x,x]\in FI(\cC)$, where $\id_x$ is the identity morphism from $\mor xx$. Observe that
\begin{align}\label{eq-e_x-alpha-e_y}
 e_x\alpha e_y=
 \begin{cases}
  \alpha_{xy}[x,y], & x\le y,\\
  0, & x\not\le y.
 \end{cases}
\end{align}

Let $(P,\preceq)$ be a preordered set and $R$ a ring. We slightly change the construction of the pocategory $\cC(P,R)$, introduced in~\cite{Khripchenko10-quasi}.

Denote by $\sim$ the natural equivalence relation on $P$, namely $x\sim y\iff x\preceq y\preceq x$, and by $\overline P$ the quotient set $P/{\sim}$. Define $\ob{\cC(P,R)}$ to be $\overline P$ with the induced partial order $\le$. For any pair $\bar x,\bar y\in\ob\cC(P,R)$:
$$
 \mor{\bar x}{\bar y}=RFM_{\bar x\times\bar y}(R),
$$
where $RFM_{I\times J}(R)$ denotes the additive group of row-finite matrices over $R$, whose rows are indexed by the elements of $I$ and columns by the elements of $J$ (for the ring $RFM_{I\times I}(R)$ we shall use the shorter notation $RFM_I(R)$). The composition of morphisms in $\cC(P,R)$ is the matrix multiplication, which is defined by the row-finiteness condition.

The {\it finitary incidence ring~\cite{Khripchenko10-quasi} of $P$ over $R$}, denoted by $FI(P,R)$, is by definition $FI(\cC(P,R))$.\footnote{We allow $\mor{\bar x}{\bar y}$ to be nonzero, when $x\not\le y$, in contrast with~\cite{Khripchenko10-quasi}. This, however, does not change the resulting ring $FI(P,R)$.}

\section{Some technical lemmas}\label{sec-tech}

Let $R$ be an arbitrary (associative) ring. An {\it idempotent} of $R$ is an element $e\in R$ with $e^2=e$. The set of idempotents of $R$ is denoted by $E(R)$. Two (distinct) idempotents are called {\it orthogonal} if they commute and their product is zero.

\begin{lem}\label{lem-ed(r)f}
 Let $d$ be a Jordan derivation of $R$ and $r\in R$. For any pair of orthogonal $e,f\in E(R)$:
 \begin{align}\label{eq-ed(r)f}
  ed(r)f=ed(erf)f-ed(e)rf-erd(f)f+ed(fre)f.
 \end{align}
 Moreover, for any $e\in E(R)$:
 \begin{align}\label{eq-ed(r)e}
  ed(r)e=ed(ere)e-ed(e)re-erd(e)e.
 \end{align}
\end{lem}
\begin{proof}
 By~\eqref{eq-d(rst+tsr)}
 \begin{align*}
  d(erf+fre)&=d(e)rf+ed(r)f+erd(f)+d(f)re+fd(r)e+frd(e).
 \end{align*}
 Multiplying this by $e$ on the left and by $f$ on the right and using the assumption that $ef=fe=0$, we come to~\eqref{eq-ed(r)f}.

 For~\eqref{eq-ed(r)e} we use~\eqref{eq-d(rsr)}:
 $$
  d(ere)=d(e)re+ed(r)e+erd(e).
 $$
\end{proof}

\begin{cor}\label{cor-ed(frf)e}
 Let $d$ be a Jordan derivation of $R$, $e,f\in E(R)$ be orthogonal and $r\in fRf$. Then $ed(r)e=0$.
\end{cor}
\noindent For $er=re=0$ in this case, so the right-hand side of~\eqref{eq-ed(r)e} is zero.

\begin{rem}\label{rem-ed(fre)f=0}
 Under the conditions of Lemma~\ref{lem-ed(r)f} if $d$ is a derivation of $R$, then $ed(fre)f=0$.
\end{rem}
\noindent Indeed, in this case $d(erf)=d(e)rf+ed(r)f+erd(f)$, and thus $ed(r)f=ed(erf)f-ed(e)rf-erd(f)f$. Comparing this with~\eqref{eq-ed(r)f}, we obtain the desired equality.

\begin{lem}\label{lem-ed(e)f+ed(f)f}
 Given a Jordan derivation $d$ of $R$, for all $e,f\in E(R)$, which are either orthogonal or equal, one has
\begin{align*}
 ed(e)f+ed(f)f=0.
\end{align*}
\end{lem}
\begin{proof}
 Suppose that $e$ and $f$ are orthogonal. Substitute $r=e$ into~\eqref{eq-ed(r)f}:
$$
 ed(e)f=ed(ef)f-ed(e)ef-ed(f)f+ed(fe)f,
$$
which is $-ed(f)f$, as $ef=fe=0$.

If $e=f$, then set $r=e$ in~\eqref{eq-ed(r)e}:
$$
 ed(e)e=ed(e)e-ed(e)e-ed(e)e=-ed(e)e.
$$
\end{proof}

Given a map $d:R\to R$ and $e\in E(R)$, denote by $d_e$ the map $eRe\to eRe$ with $d_e(r)=ed(r)e$, $r\in eRe$.

\begin{rem}\label{rem-(d_f)_e}
  Let $d$ be a map $R\to R$ and $e,f\in E(R)$ with $ef=fe=e$. Then $eRe\subseteq fRf$ and $d_e=(d_f)_e$.
\end{rem}
\noindent Indeed, $eRe=feRef\subseteq fRf$, and for any $r\in eRe$ one has
$$
d_e(r)=ed(r)e=efd(r)fe=ed_f(r)e=(d_f)_e(r).
$$

\begin{lem}\label{lem-d_e}
 Let $d$ be a derivation (respectively Jordan derivation) of $R$. Then $d_e$ is a derivation (respectively Jordan derivation) of the ring $eRe$.
\end{lem}
\begin{proof}
 We prove the assertion, when $d$ is a derivation of $R$ (the proof for a Jordan derivation is similar).

 Clearly, $d_e$ is additive. Taking $r,s\in eRe$, one sees that
 \begin{align*}
  d_e(rs)=ed(rs)e&=e(d(r)s+rd(s))e=ed(r)se+erd(s)e\\
  &=ed(r)es+red(s)e=d_e(r)s+rd_e(s).
 \end{align*}
 Here we used the obvious fact that $r$ and $s$ commute with $e$.
\end{proof}


\begin{lem}\label{lem-ed(rs)f}
 Let $d$ be a Jordan derivation of $R$, $e,f,g$ a triple of idempotents of $R$, any two distinct elements of which are orthogonal and at least two elements of which are different. Then for all $r\in eRg$ and $s\in gRf$ one has
 \begin{align}\label{eq-ed(rs)f}
  ed(rs)f=ed(r)s+rd(s)f.
 \end{align}
\end{lem}
\begin{proof}
 Suppose first that $e\ne f$. Then $sr=0$, so by~\eqref{eq-d(rs+sr)}
$$
 d(rs)=d(r)s+rd(s)+d(s)r+sd(r).
$$
To get~\eqref{eq-ed(rs)f}, it is enough to prove that
\begin{align}\label{eq-ed(s)rf=esd(r)f=0}
 ed(s)rf=esd(r)f=0.
\end{align}
If $e\ne g\ne f$, then~\eqref{eq-ed(s)rf=esd(r)f=0} follows from $rf=es=0$. Consider now the subcase $e=g\ne f$. We immediately have $ed(s)rf=0$, as $rf=0$. Furthermore, since $gs=sf$, then
\begin{align*}
 esd(r)f=gsd(r)f=sfd(r)f,
\end{align*}
which is zero by Corollary~\ref{cor-ed(frf)e}. The subcase $e\ne g=f$ is symmetric to the previous one.

It remains to consider the case $e=f\ne g$. Now $sr$ may be nonzero, however $ed(sr)f=ed(sr)e=0$ thanks to Corollary~\ref{cor-ed(frf)e}. As above,~\eqref{eq-ed(rs)f} is explained by~\eqref{eq-ed(s)rf=esd(r)f=0}, because $rf=es=0$.
 \end{proof}

\begin{cor}\label{cor-ed(rs)f}
 Let $d$ be a Jordan derivation of $R$. Then $d$ satisfies~\eqref{eq-ed(rs)f} for all triples $e,f,g\in E(R)$, any two distinct elements of which are orthogonal, if and only if $d_e$ is a derivation of $eRe$ for all $e\in E(R)$.
\end{cor}
\noindent Indeed, in view of Lemma~\ref{lem-ed(rs)f} equality~\eqref{eq-ed(rs)f} holds for all such triples $e,f,g\in E(R)$ if and only if it holds for $e=f=g\in E(R)$, that is $d_e(rs)=d_e(r)s+rd_e(s)$ with $r,s\in eRe$.

\section{Jordan derivations of a ring with certain set of pairwise orthogonal idempotents}\label{sec-tech-fin}

From now on suppose that the ring $R$ admits a set of pairwise orthogonal idempotents $E$ such that, given $r\in R$ and $e,f\in E$, the set
\begin{align}\label{eq-E'}
 E(e,r,f):=\{g\in E\mid erg\ne 0\mbox{ and }\exists s\in R: gsf\ne 0\}\mbox{ is finite.}
\end{align}
Taking additionally $s\in R$, we see that
\begin{align*}
 E(e,r,s,f):=\{g\in E\mid ergsf\ne 0\}\subseteq E(e,r,f),
\end{align*}
so $|E(e,r,s,f)|<\infty$. We shall also require that $R$ satisfies
\begin{align}\label{eq-ersf=sum-ergsf}
 ersf=\sum_{g\in E(e,r,s,f)}ergsf,
\end{align}
where the latter sum is defined to be $0$, when $E(e,r,s,f)=\emptyset$.

\begin{rem}\label{rem-E'}
 It is obvious that in~\eqref{eq-ersf=sum-ergsf} one can replace $E(e,r,s,f)$ by any finite subset of $E$ containing $E(e,r,s,f)$.
\end{rem}

\begin{lem}\label{lem-ergd(g)hsf-ne-0-finite}
 Let $d:R\to R$ be a map. Then for all $r,s\in R$ and $e,f\in E$ the set
$$
 X=\{(g,h)\in E\times E\mid ergd(g)hsf\ne 0\}
$$
is finite.
\end{lem}
\begin{proof}
 Suppose on the contrary that $X$ is infinite. Note that if $(g,h)\in X$, then $erg\ne 0$ and $gd(g)hsf\ne 0$, so $g\in E(e,r,f)$. In view of~\eqref{eq-E'} there is only a finite number of such $g$. Hence, one can find $g'\in E$ and infinitely many $h\in E$, such that $g'd(g')hsf\ne 0$. The latter implies that $E(g',d(g'),f)$ is infinite. This is impossible due to~\eqref{eq-E'}.
\end{proof}

\begin{lem}\label{lem-d'-der}
 Let $d$ be a Jordan derivation of $R$, such that for all $e\in E$ the map $d_e$ is a derivation of $eRe$. Suppose that, given $r\in R$, there exists a unique $d'(r)\in R$ with
 \begin{align}\label{eq-d'}
  ed'(r)f=ed(erf)f-ed(e)rf-erd(f)f
 \end{align}
 for all $e,f\in E$. Then the map $d':R\to R$ is a derivation of $R$.
\end{lem}
\begin{proof}
First of all observe from~\eqref{eq-d'} that
$$
 ed'(r+s)f=ed(e(r+s)f)f-ed(e)(r+s)f-e(r+s)d(f)f=e(d'(r)+d'(s))f
$$
due to additivity of $d$. By the uniqueness of $d'(r+s)$ we have $d'(r+s)=d'(r)+d'(s)$.

Fix $r,s\in R$ and $e,f\in E$. We shall show that $e(d'(r)s+rd'(s))f=ed'(rs)f$. By\cref{eq-ersf=sum-ergsf,eq-d',rem-E'} we have
\begin{align}
 ed'(r)sf&=\sum_{g\in E'}ed'(r)gsf\notag\\
 &=\sum_{g\in E'}(ed(erg)g-ed(e)rg-erd(g)g)sf,\label{eq-ed(r)sf}\\
 erd'(s)f&=\sum_{g\in E'}ergd'(s)f\notag\\
 &=\sum_{g\in E'}er(gd(gsf)f-gd(g)sf-gsd(f)f)\label{eq-erd(s)f},
\end{align}
where $E'$ is the union of finite sets $\bigcup_{i=1}^6E_i$ with
\begin{align*}
 E_1&=E(e,d'(r),s,f),\\
 E_2&=E(e,r,d'(s),f),\\
 E_3&=E(e,d(e)r,s,f),\\
 E_4&=E(e,r,sd(f),f),\\
 E_5&=E(e,r,s,f),\\
 E_6&\mbox{ being such that }\{(g,h)\in E\times E\mid ergd(g)hsf\ne 0\}\subseteq E_6\times E_6
\end{align*}
(the latter is finite by Lemma~\ref{lem-ergd(g)hsf-ne-0-finite}). Adding~\eqref{eq-ed(r)sf} and~\eqref{eq-erd(s)f}, we get
  \begin{align}
   e(d'(r)s+rd'(s))f&=\sum_{g\in E'}(ed(erg)gsf+ergd(gsf)f)\label{eq-ed(ergsf)f}\\
   &-\sum_{g\in E'}ed(e)rgsf-\sum_{g\in E'}ergsd(f)f\label{eq-ed(e)rsf-ersd(f)f}\\
   &-\sum_{g\in E'}(erd(g)gsf+ergd(g)sf).\label{eq-sum-erd(g)gsf-ergd(g)sf}
  \end{align}
  Clearly,~\eqref{eq-ed(e)rsf-ersd(f)f} is
  $$
   -\sum_{g\in E_3}ed(e)rgsf-\sum_{g\in E_4}ergsd(f)f=-ed(e)rsf-ersd(f)f.
  $$
  By Corollary~\ref{cor-ed(rs)f} the right-hand side of~\eqref{eq-ed(ergsf)f} is
  \begin{align*}
   \sum_{g\in E'}ed(ergsf)f=ed\left(\sum_{g\in E'}ergsf\right)f=ed\left(\sum_{g\in E_5}ergsf\right)f=ed(ersf)f.
  \end{align*}

By~\eqref{eq-ersf=sum-ergsf} the sum~\eqref{eq-sum-erd(g)gsf-ergd(g)sf} equals
\begin{align}\label{eq-sum-erhd(g)gsf-ergd(g)hsf}
 -\sum_{g\in E'}\left(\sum_{h\in E_7}erhd(g)gsf+\sum_{h\in E_8}ergd(g)hsf\right),
\end{align}
where $E_7$ and $E_8$ are finite subsets of $E$, more precisely,
\begin{align*}
 E_7&=\bigcup_{g\in E'}E(e,r,d(g),g),\\
 E_8&=\bigcup_{g\in E'}E(g,d(g),s,f).
\end{align*}
Setting $E''=E'\cup E_7\cup E_8$, we conclude that~\eqref{eq-sum-erhd(g)gsf-ergd(g)hsf} is
\begin{align*}
  &-\sum_{(g,h)\in E'\times E''}erhd(g)gsf-\sum_{(g,h)\in E'\times E''}ergd(g)hsf
\end{align*}
Interchanging $g$ and $h$ in the first sum and using Lemma~\ref{lem-ed(e)f+ed(f)f}, we get
\begin{align}
  &-\sum_{(g,h)\in E''\times E'}ergd(h)hsf-\sum_{(g,h)\in E'\times E''}ergd(g)hsf\notag\\
  &=\sum_{(g,h)\in E''\times E'}ergd(g)hsf-\sum_{(g,h)\in E'\times E''}ergd(g)hsf.\label{eq-sum-ergd(g)hsf}
 \end{align}
Since
$$
 \{(g,h)\in E\times E\mid ergd(g)hsf\ne 0\}\subseteq E_6\times E_6\subseteq E'\times E'\subseteq(E'\times E'')\cap(E''\times E'),
$$
it follows that~\eqref{eq-sum-ergd(g)hsf} equals
\begin{align*}
  \sum_{(g,h)\in E_6\times E_6}ergd(g)hsf-\sum_{(g,h)\in E_6\times E_6}ergd(g)hsf=0.
\end{align*}

Thus,
  \begin{align*}
   e(d'(r)s+rd'(s))f=ed(ersf)f-ed(e)rsf-ersd(f)f,
  \end{align*}
  the latter being $ed'(rs)f$ by~\eqref{eq-d'}. In view of the uniqueness of $d'(rs)$, one has $d'(rs)=d'(r)s+rd'(s)$. Since $r$ and $s$ were arbitrary elements of $R$, the map $d'$ is a derivation.
 \end{proof}

\section{Jordan derivations of \texorpdfstring{$FI(\cC)$}{FI(C)}}\label{sec-FI(C)}


We first specify Lemma~\ref{lem-ed(r)f} for $FI(\cC)$ and the idempotents $e_x,e_y\in FI(\cC)$.
\begin{lem}\label{lem-d-alpha_xy}
 Let $d$ be a Jordan derivation of $FI(\cC)$. Then for any $\alpha\in FI(\cC)$ and for all $x\le y$:
 \begin{align*}
  d(\alpha)_{xy}=d(\alpha_{xy}[x,y])_{xy}-(d(e_x)\alpha)_{xy}-(\alpha d(e_y))_{xy}.
 \end{align*}
\end{lem}
\begin{proof}
 Suppose first that $x<y$. Then the idempotents $e_x$ and $e_y$ are orthogonal, so by~\eqref{eq-ed(r)f}
 \begin{align*}
  e_xd(\alpha)e_y=e_xd(e_x\alpha e_y)e_y-e_xd(e_x)\alpha e_y-e_x\alpha d(e_y)e_y+e_xd(e_y\alpha e_x)e_y.
 \end{align*}
 Note that $x<y$ means that $x\le y$ and $y\not\le x$. It remains to apply~\eqref{eq-e_x-alpha-e_y}.

 When $x=y$, the result follows from~\eqref{eq-ed(r)e} with $e=e_x$ and~\eqref{eq-e_x-alpha-e_y}.
\end{proof}

%

\begin{lem}\label{lem-FI(C)-satisfies-fin-cond}
 The ring $FI(\cC)$ satisfies\cref{eq-E',eq-ersf=sum-ergsf} with $E$ being $\{e_x\}_{x\in\ob\cC}$.
\end{lem}
\begin{proof}
 Clearly, $E$ is a set of pairwise orthogonal idempotents of $FI(\cC)$. Let $\alpha\in FI(\cC)$ and $x,y\in\ob\cC$.

If $x\not\le y$, then for any $z\in\ob\cC$ either $x\not\le z$, or $z\not\le y$. In the first case $e_x\alpha e_z=0$, and in the second one $e_z\beta e_y=0$ for any $\beta\in FI(\cC)$, thanks to~\eqref{eq-e_x-alpha-e_y}. Hence $E(e_x,\alpha,e_y)=\emptyset$, so~\eqref{eq-E'} holds. As $e_x\alpha\beta e_y=0$ by~\eqref{eq-e_x-alpha-e_y}, equality~\eqref{eq-ersf=sum-ergsf} also takes place.

If $x\le y$, then $e_x\alpha e_z\ne 0$ means that $\alpha_{xz}\ne 0_{xz}$ (in particular, $x\le z$) and $\exists\beta\in FI(\cC): e_z\beta e_y\ne 0$ implies that $z\le y$. Therefore,
$$
 \{z\in\ob\cC\mid e_z\in E(e_x,\alpha,e_y)\}\subseteq\{z\in[x,y]\mid\alpha_{xz}\ne 0_{xz}\},
$$
the latter set being finite due to the fact that $\alpha$ is a finitary series, whence~\eqref{eq-E'}. As to~\eqref{eq-ersf=sum-ergsf}, observe from~\eqref{eq-e_x-alpha-e_y} that $e_x\alpha e_z\beta e_y=\alpha_{xz}\beta_{zy}[x,y]$, so
$$
 Z:=\{z\in\ob\cC\mid e_z\in E(e_x,\alpha,\beta,e_y)\}=\{z\in[x,y]\mid \alpha_{xz}\beta_{zy}\ne 0_{xy}\}.
$$
Moreover, by~\eqref{eq-convolution}
\begin{align*}
 e_x\alpha\beta e_y=(\alpha\beta)_{xy}[x,y]&=\left(\sum_{x\le z\le y}\alpha_{xz}\beta_{zy}\right)[x,y]=\left(\sum_{z\in Z}\alpha_{xz}\beta_{zy}\right)[x,y]\\
 &=\sum_{z\in Z}(\alpha_{xz}[x,z]\cdot\beta_{zy}[z,y])=\sum_{e_z\in E(e_x,\alpha,\beta,e_y)}e_x\alpha e_z\beta e_y,
\end{align*}
proving~\eqref{eq-ersf=sum-ergsf}.
\end{proof}

Let $d:FI(\cC)\to FI(\cC)$ and $x\in\ob\cC$. Observe that the ring $e_xFI(\cC)e_x$ may be identified with $\mor xx$ by $e_x\alpha e_x=\alpha_{xx}[x,x]\leftrightarrow\alpha_{xx}$. Denote by $d_x$ the map $\mor xx\to\mor xx$, which corresponds to $d_{e_x}:e_xFI(\cC)e_x\to e_xFI(\cC)e_x$ under this identification, namely, $d_x(\varphi)=d(\varphi[x,x])_{xx}$ for $\varphi\in\mor xx$.

\begin{prop}\label{prop-d-der-iff-d_x-der}
 Let $d$ be a Jordan derivation of $FI(\cC)$. Then $d$ is a derivation of $FI(\cC)$ if and only if for all $x\in\ob\cC$ the map $d_x$ is a derivation of $\mor xx$.
\end{prop}
\begin{proof}
The necessity is explained by Lemma~\ref{lem-d_e}.

For the sufficiency consider~\eqref{eq-d'} with $e=e_x$, $f=e_y$ and $r=\alpha\in FI(\cC)$:
$$
 d'(\alpha)_{xy}=d(\alpha_{xy}[x,y])_{xy}-(d(e_x)\alpha)_{xy}-(\alpha d(e_y))_{xy}.
$$
The right-hand side is $d(\alpha)_{xy}$ by Lemma~\ref{lem-d-alpha_xy}. Therefore, $d'=d$, in particular, $d'$ is a well-defined function $FI(\cC)\to FI(\cC)$. It remains to apply Lemmas~\ref{lem-d'-der} and~\ref{lem-FI(C)-satisfies-fin-cond}.
 \end{proof}

An element of a preordered set will be said to be {\it isolated}, whenever it is incomparable with any other element of this set. In particular, an object $x\in\ob\cC$ is {\it isolated}, if it is isolated under $\le$.

\begin{lem}\label{lem-tilde-d_x}
 Let $x\in\ob\cC$ be isolated. For any map $d:\mor xx\to\mor xx$ and $\alpha\in FI(\cC)$ define
 \begin{align}\label{eq-tilde-d}
  \tilde d(\alpha)=d(\alpha_{xx})[x,x]+\alpha-\alpha_{xx}[x,x].
 \end{align}
 Then $\tilde d$ is a derivation (respectively Jordan derivation) of $FI(\cC)$ if and only if $d$ is a derivation (respectively Jordan derivation) of $\mor xx$.
\end{lem}
\begin{proof}
Clearly, $\tilde d(\alpha)$ is a finitary series, as $\tilde d(\alpha)_{uv}=\alpha_{uv}$ for all $u<v$.

Since $x$ is incomparable with the rest of the objects of $\cC$, then $FI(\cC)=FI(\{x\})\oplus FI(\cC\setminus\{x\})$. Here $\{x\}$ is the full subcategory of $\cC$ with the unique object $x$ and $\cC\setminus\{x\}$ denotes the subcategory obtained from $\cC$ by removing $x$ as well as all the morphisms from $x$ and into $x$. It remains to note that $FI(\{x\})\cong\mor xx$ and $\tilde d=d\oplus\id$ up to this isomorphism.
\end{proof}

Let $M$ be a unital $(R,S)$-bimodule. Following~\cite{Zhang-Yu06} by the {\it triangular ring} $Tri(R,M,S)$ we mean the ring whose additive group is $R\oplus M\oplus S$ and the multiplication is defined by
$$
 (r,m,s)(r',m',s')=(rr',rm'+ms',ss')
$$
(the same construction can also be found in~\cite{Anderson-Winders09}). The next remark is straightforward.
\begin{rem}\label{rem-idealization}
 Let with $x<y$. Then $\mor xy$ is a (unital) $(\mor xx,\mor yy)$-bimodule, such that
 $$
 Tri(\mor xx,\mor xy,\mor yy)\cong(e_x+e_y)FI(\cC)(e_x+e_y).
 $$
\end{rem}
\noindent Indeed, the isomorphism is given by $(\varphi,\psi,\eta)\leftrightarrow\varphi[x,x]+\psi[x,y]+\eta[y,y]$, where $\varphi\in\mor xx$, $\psi\in\mor xy$ and $\eta\in\mor yy$.\\

We say that an $(R,S)$-bimodule is {\it faithful on the left (respectively on the right)}, if it is faithful as a left $R$-module (respectively as a right $S$-module).
\begin{lem}\label{lem-triangular}
 Let $x<y$ be such that $\mor xy$ is faithful on the left and on the right. If $d$ is a Jordan derivation of $FI(\cC)$, then $d_x$ and $d_y$ are derivations of $\mor xx$ and $\mor yy$ respectively.
\end{lem}
\begin{proof}
By Remark~\ref{rem-idealization} the ring $(e_x+e_y)FI(\cC)(e_x+e_y)$ is triangular and by Lemma~\ref{lem-d_e} the map $d_{e_x+e_y}$ is a Jordan derivation of $(e_x+e_y)FI(\cC)(e_x+e_y)$. Hence, $d_{e_x+e_y}$ is a derivation of $(e_x+e_y)FI(\cC)(e_x+e_y)$ in view of~\cite[Theorem 2.1]{Zhang-Yu06}\footnote{In~\cite{Zhang-Yu06} the authors study Jordan derivations of triangular algebras over commutative rings of characteristics different from $2$. However, their proofs are based on formulas~\cref{eq-d(r^2),eq-d(rsr),eq-d(rs+sr),eq-d(rst+tsr)} and do not use the vector space structure, so the main result is true for Jordan derivations (in the sense of~\cite{Jacobson-Rickart50}) of triangular rings.} Then $d_{e_x}$, being $(d_{e_x+e_y})_{e_x}$ thanks to Remark~\ref{rem-(d_f)_e}, is a derivation of $e_xFI(\cC)e_x$ by the same Lemma~\ref{lem-d_e}. Similarly $d_{e_y}=(d_{e_x+e_y})_{e_y}$ is a derivation of $e_yFI(\cC)e_y$. As $d_x$ and $d_y$ may be identified with $d_{e_x}$ and $d_{e_y}$,
 this completes the proof.
\end{proof}

\begin{thrm}\label{thrm-crit-jder-FI(C)-is-der}
 Assume that for any non-isolated $x\in\ob\cC$ there is $y>x$ (or $y<x$), such that $\mor xy$ (or $\mor yx$) is faithful on the left and on the right. Then each Jordan derivation of $FI(\cC)$ is a derivation if and only if for any isolated $x\in\ob\cC$ each Jordan derivation of $\mor xx$ is a derivation.
\end{thrm}
\begin{proof}
 Suppose that each Jordan derivation of $FI(\cC)$ is a derivation. Take an isolated $x\in\ob\cC$ and a Jordan derivation $d$ of $\mor xx$. By Lemma~\ref{lem-tilde-d_x} the map $\tilde d$ defined by~\eqref{eq-tilde-d} is a Jordan derivation of $FI(\cC)$ and hence a derivation of $FI(\cC)$ by the assumption. By the same lemma $d$ is a derivation of $\mor xx$.

 Conversely, assume that, given an isolated $x\in\ob\cC$, each Jordan derivation of $\mor xx$ is a derivation. Let $d$ be a Jordan derivation of $FI(\cC)$ and $x\in\ob\cC$. If $x$ is isolated, then $d_x$ is a derivation of $\mor xx$ by the assumption. Otherwise there exists $y>x$ (or $y<x$), such that $\mor xy$ (or $\mor yx$) is faithful on the left and on the right. Hence $d_x$ is a derivation of $\mor xx$ thanks to Lemma~\ref{lem-triangular}. Consequently, $d$ is a derivation of $FI(\cC)$ in view of Proposition~\ref{prop-d-der-iff-d_x-der}.
\end{proof}


\section{Jordan derivations of \texorpdfstring{$RFM_I(R)$}{RFM(R)}}\label{sec-RFM}

Let $R$ be a unital ring and $I$ a set. Denote by $e_{ij}$, $i,j\in I$, the matrix from $RFM_I(R)$ having a unique nonzero element equal to $1$ at the intersection of the $i$-th row and $j$-th column. Then $\{e_{ii}\}_{i\in I}$ is a set of pairwise orthogonal idempotents of $RFM_I(R)$. Given a matrix $\alpha=\{\alpha_{ij}\}_{i,j\in I}\in RFM_I(R)$, the product $e_{ii}\alpha e_{jj}$ is clearly $\alpha_{ij}e_{ij}$.

We first reformulate Lemma~\ref{lem-ed(r)f} in these notations.
\begin{lem}\label{lem-d'(r)-is-row-finite}
  Let $d$ be a Jordan derivation of $RFM_I(R)$. For any $\alpha\in RFM_I(R)$ and $i\ne j\in I$:
 \begin{align}\label{eq-d(alpha)_ij}
  d(\alpha)_{ij}=d(\alpha_{ij}e_{ij})_{ij}-(d(e_{ii})\alpha)_{ij}-(\alpha d(e_{jj}))_{ij}+d(\alpha_{ji}e_{ji})_{ij}.
 \end{align}
 Moreover,
 \begin{align}\label{eq-d(alpha)_ii}
  d(\alpha)_{ii}=d(\alpha_{ii}e_{ii})_{ii}-(d(e_{ii})\alpha)_{ii}-(\alpha d(e_{ii}))_{ii}.
 \end{align}
\end{lem}

\begin{lem}\label{lem-RFM_I-times-I(R)-satisfies-fin-cond}
 The ring $RFM_I(R)$ satisfies\cref{eq-E',eq-ersf=sum-ergsf} with $E=\{e_{ii}\}_{i\in I}$.
\end{lem}
\begin{proof}
Given $\alpha\in RFM_I(R)$ and $i,j\in I$, one clearly has
$$
 \{k\in I\mid e_{kk}\in E(e_{ii},\alpha,e_{jj})\}
 \subseteq\{k\in I\mid e_{ii}\alpha e_{kk}\ne 0\}
 =\{k\in I\mid \alpha_{ik}\ne 0\},
$$
which is finite due to the row-finiteness condition. This proves~\eqref{eq-E'}. Moreover, for any $\beta\in RFM_I(R)$:
\begin{align*}
 I'&=\{k\in I\mid e_{kk}\in E(e_{ii},\alpha,\beta,e_{jj})\}=\{k\in I\mid \alpha_{ik}\beta_{kj}\ne 0\},
\end{align*}
so
$$
 e_{ii}\alpha\beta e_{jj}=(\alpha\beta)_{ij}e_{ij}=\left(\sum_{k\in I'}\alpha_{ik}\beta_{kj}\right)e_{ij}=\sum_{k\in I'}e_{ii}\alpha e_{kk}\beta e_{jj},
$$
giving~\eqref{eq-ersf=sum-ergsf}.
\end{proof}

For each $i\in I$ the ring $e_{ii}RFM_I(R)e_{ii}=Re_{ii}$ is obviously isomorphic to $R$. If $d$ is a map from $RFM_I(R)$ to itself and $i\in I$, then $d_i$ denotes the induced map $R\to R$, namely, $d_i(r)=d(re_{ii})_{ii}$.

Let us write~\eqref{eq-d'} for $e=e_{ii}$, $f=e_{jj}$ and $r=\alpha\in RFM_I(R)$:
\begin{align}\label{eq-d'(alpha)_ij}
   d'(\alpha)_{ij}=d(\alpha_{ij}e_{ij})_{ij}-(d(e_{ii})\alpha)_{ij}-(\alpha d(e_{jj}))_{ij}.
\end{align}
This defines a (unique) matrix $d'(\alpha)$ over $R$ whose entries are indexed by the pairs of elements of $I$. It is not clear, however, whether $d'(\alpha)$ is row-finite.

\begin{lem}\label{lem-d'(alpha)-is-row-finite}
 Let $d$ be a Jordan derivation of $RFM_I(R)$ and $\alpha\in RFM_I(R)$. Then $d'(\alpha)\in RFM_I(R)$.
\end{lem}
\begin{proof}
 Suppose on the contrary that there are $i\in I$ and an infinite $J\subseteq I$, such that $d'(\alpha)_{ij}\ne 0$ for $j\in J$. According to~\eqref{eq-d'(alpha)_ij} at least one of the sets
 $$
  A=\{d(\alpha_{ij}e_{ij})_{ij}\}_{j\in J},\ \ B=\{(d(e_{ii})\alpha)_{ij}\}_{j\in J},\ \ C=\{(\alpha d(e_{jj}))_{ij}\}_{j\in J}
 $$
should have infinitely many nonzero elements.

Since $\alpha$ is row-finite, then there is a finite $J'\subseteq J$ with $\alpha_{ij'}=0$ for $j'\in J\setminus J'$. Therefore, $d(\alpha_{ij'}e_{ij'})=0$ for such $j'$, so the set of nonzero elements of $A$ is finite.

Furthermore, the elements of $B$ belong to the $i$-th row of the fixed (in the sense that it does not depend on $j$) row-finite matrix $d(e_{ii})\alpha$. Hence, $B$ has only finitely many nonzero elements.

Suppose that the set of nonzero elements of $C$ is infinite. For simplicity assume that all the elements of $C$ are nonzero. As $(\alpha d(e_{jj}))_{ij}=\sum_{k\in I}\alpha_{ik}d(e_{jj})_{kj}$, for each $j\in J$ there is $k_j\in I$, such that both $\alpha_{ik_j}$ and $d(e_{jj})_{k_jj}$ are nonzero. But $\alpha$ is row-finite, so there is only a finite number of different indices $k_j$. Hence, one can find $k'$ among them, such that $k'=k_j$ for infinitely many $j\in J$. Consequently, $d(e_{jj})_{k'j}\ne 0$ for infinitely many $j\in J$. However, by Lemma~\ref{lem-ed(e)f+ed(f)f} with $e=e_{k'k'}$ and $f=e_{jj}$ one has $0\ne d(e_{jj})_{k'j}=-d(e_{k'k'})_{k'j}$. We get a contradiction with the fact that $d(e_{k'k'})$ is row-finite.
\end{proof}

\begin{cor}\label{cor-d'-is-der}
 Let $d$ be a Jordan derivation of $RFM_I(R)$, such that for each $i\in I$ the map $d_i$ is a derivation of $R$. Then $d'$ is a derivation of $RFM_I(R)$.
\end{cor}
\noindent This immediately follows from \cref{lem-d'-der,lem-RFM_I-times-I(R)-satisfies-fin-cond,lem-d'(alpha)-is-row-finite}.

\begin{lem}\label{lem-d-der-iff-d_i-der}
 Let $d$ be a Jordan derivation of $RFM_I(R)$. Then $d$ is a derivation of $RFM_I(R)$ if and only if for all $i\in I$ the map $d_i$ is a derivation of $R$.
\end{lem}
\begin{proof}
 The necessity is Lemma~\ref{lem-d_e}.

 For the sufficiency we use the idea of the proof of~\cite[Theorem~3.3]{XiaoJDerIAlg}. When $|I|=1$, the assertion is trivial. Suppose that $|I|>1$ and $d_i$ is a derivation of $R$ for each $i\in I$. By Corollary~\ref{cor-d'-is-der} the map $d'$ is a derivation of $RFM_I(R)$. Therefore, $d-d'$ is a Jordan derivation of $RFM_I(R)$. Note from\cref{eq-d(alpha)_ij,eq-d(alpha)_ii,eq-d'(alpha)_ij} that
 \begin{align}
  (d-d')(\alpha)_{ij}&=d(\alpha_{ji}e_{ji})_{ij},\ \ i\ne j,\label{eq-(d-d')(alpha)_ij}\\
  (d-d')(\alpha)_{ii}&=0.\label{eq-(d-d')(alpha)_ii}
 \end{align}

 Fix a pair of distinct $i,j\in I$ (such a pair exists, as $|I|>1$). Denote by $A_{ij}$ the $R$-submodule of $RFM_I(R)$ generated by $e_{ii},e_{ij},e_{ji},e_{jj}$. It is clearly a subring of $RFM_I(R)$ isomorphic to the ring of $2\times 2$ matrices over $R$. Let $\alpha\in A_{ij}$ and $k\ne l$ with $\{k,l\}\ne\{i,j\}$. One sees from~\eqref{eq-(d-d')(alpha)_ij} that $(d-d')(\alpha)_{kl}=0$, as $\alpha_{lk}=0$. Hence, $(d-d')(\alpha)\in A_{ij}$ and thus $A_{ij}$ is invariant under $d-d'$. It follows that $(d-d')|_{A_{ij}}$ is a Jordan derivation of $A_{ij}$ and hence a derivation of $A_{ij}$ in view of Theorems 8 and 22 from~\cite{Jacobson-Rickart50}.

 For any $\alpha\in RFM_I(R)$ set
 \begin{align}\label{eq-tilde-alpha}
  \tilde\alpha=\alpha_{ii}e_{ii}+\alpha_{ij}e_{ij}+\alpha_{ji}e_{ji}+\alpha_{jj}e_{jj}\in A_{ij}.
 \end{align}
 Since $(d-d')|_{A_{ij}}$ is a derivation of $A_{ij}$, by Remark~\ref{rem-ed(fre)f=0} we have $(d-d')(\tilde\alpha_{ji}e_{ji})_{ij}=0$. But
 $$
  (d-d')(\tilde\alpha_{ji}e_{ji})_{ij}=d(\tilde\alpha_{ji}e_{ji})_{ij}=d(\alpha_{ji}e_{ji})_{ij}=(d-d')(\alpha)_{ij}
 $$
 thanks to~\eqref{eq-(d-d')(alpha)_ij} and~\eqref{eq-tilde-alpha}. So, $(d-d')(\alpha)_{ij}=0$. As $i$ and $j$ were arbitrary (distinct) elements of $I$, together with~\eqref{eq-(d-d')(alpha)_ii} this means that $(d-d')(\alpha)=0$. Thus, $d=d'$ and hence it is a derivation.
\end{proof}

\begin{lem}\label{lem-d_i-is-der}
 Let $d$ be a Jordan derivation of $RFM_I(R)$. If $|I|>1$ and $i\in I$, then $d_i$ is a derivation of $R$.
\end{lem}
\begin{proof}
 As $|I|>1$, there exists $j\ne i$ in $I$. Let $A_{ij}$ denote the subring of $RFM_I(R)$ from the proof of Lemma~\ref{lem-d-der-iff-d_i-der}. Note that $A_{ij}=eRFM_I(R)e$, where $e$ is the idempotent $e_{ii}+e_{jj}$. By Lemma~\ref{lem-d_e} the map $d_e$ is a Jordan derivation of $A_{ij}$. Since $A_{ij}$ is isomorphic to the ring of $2\times 2$ matrices over $R$, then $d_e$ is a derivation of $A_{ij}$ according to Theorems 8 and 22 from~\cite{Jacobson-Rickart50}.

 Now since $e_{ii}e=ee_{ii}=e_{ii}$, thanks to Remark~\ref{rem-(d_f)_e} we have $e_{ii}RFM_I(R)e_{ii}\subseteq A_{ij}$ and $d_{e_{ii}}=(d_e)_{e_{ii}}$, so $d_{e_{ii}}$ is a derivation of $e_{ii}RFM_I(R)e_{ii}$ by Lemma~\ref{lem-d_e}. The latter means that $d_i$ is a derivation of $R$.
\end{proof}

\begin{thrm}\label{thrm-jder-of-RFM_I-for-|I|>1}
 Let $R$ be a ring and $I$ a set. If $|I|>1$, then each Jordan derivation of $RFM_I(R)$ is a derivation.
\end{thrm}
\begin{proof}
 A direct consequence of Lemmas~\ref{lem-d-der-iff-d_i-der} and~\ref{lem-d_i-is-der}.
\end{proof}

\begin{rem}\label{rem-jder-of-RFM_I-for-|I|=1}
 If $|I|=1$, then $RFM_I(R)\cong R$, so the question whether a Jordan derivation of $RFM_I(R)$ is a derivation reduces to the same question for the ring $R$.
\end{rem}

\section{Jordan derivations of \texorpdfstring{$FI(P,R)$}{FI(P,R)}}\label{sec-FI(P,R)}

\begin{lem}\label{lem-RFM_I-times-J-faithful}
 Let $R$ be a unital ring and $I$, $J$ (nonempty) sets. Then $RFM_{I\times J}(R)$ is faithful as a left $RFM_I(R)$-module and as a right $RFM_J(R)$-module.
\end{lem}
\begin{proof}
 We prove the assertion for the left module, the proof for the right module is similar.

 Let $\alpha\in RFM_I(R)$. Suppose that $\alpha\beta=0$ for any $\beta\in RFM_{I\times J}(R)$. Fix $i,i'\in I$. Choose an arbitrary $j\in J$ and consider the matrix $e_{i'j}\in RFM_{I\times J}(R)$. By our assumption $\alpha e_{i'j}=0$. In particular, $0=(\alpha e_{i'j})_{ij}=\alpha_{ii'}$. Since $i$ and $i'$ were arbitrary elements of $I$, the latter means that $\alpha=0$.
\end{proof}


\begin{thrm}\label{thrm-crit-jder-FI(P,R)-is-der}
Let $(P,\preceq)$ be a preordered set and $R$ a unital ring. If $P$ has no isolated elements, then every Jordan derivation of $FI(P,R)$ is a derivation. Otherwise every Jordan derivation of $FI(P,R)$ is a derivation if and only if every Jordan derivation of $R$ is a derivation.
\end{thrm}
\begin{proof}
 By Theorem~\ref{thrm-crit-jder-FI(C)-is-der} and Lemma~\ref{lem-RFM_I-times-J-faithful} every Jordan derivation of $FI(P,R)$ is a derivation if and only if for any isolated $\bar x$ of $\overline P$ every Jordan derivation of $RFM_{\bar x}(R)$ is a derivation.

 Let $\bar x$ be an isolated element of $\overline P$. If $|\bar x|>1$, then each Jordan derivation of $RFM_{\bar x}(R)$ is a derivation thanks to Theorem~\ref{thrm-jder-of-RFM_I-for-|I|>1}. If $|\bar x|=1$, then $RFM_{\bar x}(R)\cong R$, so each Jordan derivation of $RFM_{\bar x}(R)$ is a derivation if and only if each Jordan derivation of $R$ is a derivation.

 It remains to note that the isolated elements of $P$ are exactly those $x\in P$, for which $\bar x$ is isolated in $\overline P$ and $|\bar x|=1$.
\end{proof}

The next Remark generalizes Theorem~3.3 from~\cite{XiaoJDerIAlg}.
\begin{rem}\label{rem-R-linear}
 Let $(P,\preceq)$ be a preordered set and $R$ a commutative unital ring. Then every $R$-linear Jordan derivation of $FI(P,R)$ is a derivation.
\end{rem}
\noindent Indeed, if $d$ is an $R$-linear Jordan derivation of $FI(P,R)$, then $d_{\bar x}$ is an $R$-linear Jordan derivation of $RFM_{\bar x}(R)$ for all $\bar x\in\ob\cC$. If $|\bar x|>1$, then $d_{\bar x}$ is a derivation by \cref{thrm-jder-of-RFM_I-for-|I|>1}. If $|\bar x|=1$, then $d_{\bar x}$ may be identified with an $R$-linear Jordan derivation of $R$, which is trivial, as it maps $1$ to $0$ by~\eqref{eq-d(r^2)}. Thus, $d$ is a derivation 	by \cref{prop-d-der-iff-d_x-der}.

\section*{Acknowledgements}
The author is grateful to FAPESP of Brazil for the financial support (process number: 2012/01554--7) and to professor Zhankui Xiao for the valuable comments that helped to improve the article.

\bibliography{bibl}{}
\bibliographystyle{acm}

\end{document}